\documentclass{amsart}
\usepackage{verbatim,amssymb,amsmath,amscd,latexsym,amsbsy,mathrsfs} \input{xy}
\xyoption{all}

\newtheorem{lemma}{Lemma}[section]
\newtheorem{theorem}[lemma]{Theorem}
\newtheorem{corollary}[lemma]{Corollary}

\newtheorem{proposition}[lemma]{Proposition}
\theoremstyle{remark}

\theoremstyle{definition}

\DeclareMathOperator*{\spec}{Spec} 
\DeclareMathOperator*{\card}{card} 
 
\DeclareMathOperator*{\Gal}{Gal}

\newcommand{\ZZ}{\mathbb{Z}} \newcommand{\Aff}{\mathbb{A}}
\newcommand{\PP}{\mathbb{P}}

 \newcommand{\g}{\mathrm{g}}

\newcommand{\onto}{\twoheadrightarrow}

\renewcommand{\phi}{\varphi}

\def\normal{\triangleleft}

\def \PP {{\mathbb P}}

\newcounter{Lcounter}
\renewcommand{\theLcounter}{\theequation\alph{Lcounter}}

\let\oldmarginpar\marginpar
\renewcommand\marginpar[1]{\-\oldmarginpar[\raggedleft\footnotesize #1]%
{\raggedright\footnotesize #1}}

\begin{document}

\title{Subgroup structure of fundamental groups in positive characteristic }%

\author{Lior Bary-Soroker}%
\address{Instituts f\"ur Experimentelle Mathematik, Universit\"at Duisburg-Essen}
\email{barylior@post.tau.ac.il}
\author{Manish Kumar}
\address{Universit\"at Duisburg-Essen,
Fakult\"at f\"ur Mathematik}
\email{manish.kumar@uni-due.de}

\date{\today}%
\begin{abstract}
Let $\Pi$ be the \'etale fundamental group of a smooth affine curve over an algebraically closed field of characteristic $p>0$. We establish a criterion for profinite freeness of closed subgroups of $\Pi$.  Roughly speaking, if a closed subgroup of $\Pi$ is ``captured'' between two normal subgroups, then it is free, provided it contains most of the open subgroups of index $p$. In the proof we establish a strong version of ``almost $\omega$-freeness'' of $\Pi$ and then apply the Haran-Shapiro induction. 
\end{abstract}
\maketitle

\section{Introduction}
The \'etale fundamental group of a variety over a field of characteristic $p>0$ is a mysterious group. Even in the case of smooth affine curves it is not yet completely understood. The prime-to-$p$ part of the group is understood because of Grothendieck's Riemann Existence Theorem \cite[XIII, Corollary 2.12]{SGA1}. The main difficulty arises from the wildly ramified covers.

Let $C$ be a smooth affine curve over an algebraically closed field $k$ of characteristic $p>0$. Let $X$ be the smooth completion of $C$, $g$ the genus of $X$ and $r=\card(X\smallsetminus C)$.
Let $\pi_1(C)$ be the \'etale fundamental group of $C$. This is a profinite group and henceforth the terminology should be understood in the category of profinite groups, e.g.\ subgroups are closed, homomorphisms are continuous, etc.
 
Abhyankar's conjecture which was proved by Raynaud \cite{Raynaud1994} and Harbater \cite{Harbater1994} classifies the finite quotients of $\pi_1(C)$, namely a finite group $G$ is a quotient of $\pi_1(C)$ if and only if $G/p(G)$ is generated by $2g+r-1$ elements. Here $p(G)$ is the subgroup of $G$ generated by all the $p$-sylow subgroups of $G$.
In particular, $\pi_1(C)$ is not finitely generated, hence it is not determined by the set of its finite quotients.

Recently people have tried to understand the structure of $\pi_1(C)$ by studying its subgroups. In \cite{Kumar2008} the second author shows that the commutator subgroup is free of countable rank, provided $k$ is countable.
In \cite{PachecoStevensonZalesskii2009} Pacheco, Stevenson and Zalesskii make an attempt to understand the normal subgroups $N$ of $\pi_1(C)$, 
again when $k$ is countable. In \cite{HarbaterStevenson2011} Harbater and Stevenson show that $\pi_1(C)$ is almost $\omega$-free, in the sense that every finite embedding problem has a solution after restricting to an open subgroup (see also \cite{Jarden2011}). 
If $k$ is uncountable, the result of \cite{Kumar2008} is extended in \cite{Kumar2009}. We prove a diamond theorem for $\pi_1(C)$ when $k$ is countable, Theorem~\ref{thm:main} below, by using Theorem~\ref{thm:pi_1} below which strengthens \cite[Theorem~6]{HarbaterStevenson2011}. 

Let us explain the main result of this paper in detail. 
We say that a subgroup $H$ of a profinite group $\Pi$ lies in a $\Pi$-diamond if there exist $M_1,M_2\lhd \Pi$ such that $M_1\cap M_2\leq H$, but $M_1\not\leq H$ and $M_2\not\leq H$. Haran's diamond theorem states that a subgroup $H$ of a free profinite group of infinite rank that lies in a diamond is free \cite{Haran1999a}. In \cite{Bary-Soroker2006} the first author extends the diamond theorem to free profinite groups of finite rank $\geq 2$. Many subgroups lie in a diamond, notably, the commutator and any proper open subgroup of a normal subgroup. 

The diamond theorem is general in the sense that most of the other criteria follow from it, and it applies also to non-normal subgroups, in contrast to Melnikov's theorem. Moreover, the diamond theorem can be extended to Hilbertian fields \cite{Haran1999b} and to other classes of profinite groups. 
 In \cite{Bary-SorokerHaranHarbater2010} Haran, Harbater, and the first author prove a diamond theorem for semi-free profinite groups, and in \cite{Bary-SorokerStevensonZalesskii2010} Stevenson, Zalesskii and the first author establish a diamond theorem for $\pi_1(C)$, where $C$ is a projective curve of genus at least $2$ over an algebraically closed field of characteristic $0$. 

For each $g\ge 0$, we define a subgroup of $\pi_1(C)$,
$$P_g(C)=\bigcap\{\pi_1(Z)\mid Z\to C \text{ is \'etale $\ZZ/p\ZZ$-cover and the genus of }Z\ge g \}.$$ 
We note that $P_0(C)$ is the intersection of all open normal subgroups of index $p$ and we have $P_{g+1}(C)\geq P_g(C)$.

Our main result is the following diamond theorem for subgroups of $\pi_1(C)$ that are contained in $P_g(C)$ for some $g\geq 0$.

\begin{theorem}\label{thm:main}
 Let $k$ be a countable algebraically closed field of characteristic $p>0$, let $C$ be a smooth affine $k$-curve, let $\Pi=\pi_1(C)$, let $g$ be a non-negative integer, and let $M$ be a subgroup of $\Pi$. Assume $M\leq P_g(C)$ and there exist normal subgroups $M_1,M_2$ of $\Pi$ such that $M$ contains $M_1\cap M_2$ but contain neither $M_1$ or $M_2$. Then:
 \begin{enumerate}\renewcommand{\theenumi}{\roman{enumi}}
 \item For every finite simple group $S$ the direct power $S^\infty$ is a quotient of $M$. \label{part:main_a}
 \item \label{part:main_b} Assume further that $[M M_i:M] \neq p$ for $i=1,2$. Then $M$ is free of countable rank.  
 \end{enumerate}
\end{theorem}

This generalizes \cite[Theorem 4.8]{Kumar2008}. Let $\Pi'$ be the commutator subgroup of $\Pi$. Then $\Pi'\leq P_0(C)$, so $\Pi/\Pi'$ is a non finitely generated  abelian profinite group, hence $\Pi/\Pi' = A\times B$ with $A,B$ infinite. So $\Pi'$ is the intersection of the preimages of $A$ and $B$ in $\Pi$, hence lies in a diamond as in the above theorem. Since $A$ and $B$ are infinite, the hypothesis of Theorem \ref{thm:main}(\ref{part:main_b}) also holds for $\Pi'$.

We note that Theorem~\ref{thm:main} implies that for every $g\geq 0$ the subgroup $P_g(C)$ is free (cf.\ \cite{Kumar2009}). However we do not know whether Theorem~\ref{thm:main} follows from this latter assertion, because it is not clear that $M$ lies in a $P_g(C)$-diamond even if it contained in a $\pi_1(C)$-diamond.

Although a normal subgroup $N$ of $\pi_1(C)$ of infinite index is not necessarily free, every proper open subgroup of $N$ is free, provided it is contained in $P_g(C)$ for some $g$ (cf.\ Lubotzky-v.d.~Dries' theorem \cite[Proposition 24.10.3]{FriedJarden2008}). 
\begin{corollary}\label{cor:LD}
 Let $k$ be a countable algebraically closed field of characteristic $p>0$, let $C$ be a smooth affine $k$-curve, let $\Pi=\pi_1(C)$. Let $N$ be a normal subgroup of $\Pi$ of infinite index and let $M$ be a proper open subgroup of $N$ that is contained in $P_g(C)$ for some $g\geq 0$. Then $M$ is free. 
\end{corollary}

A stronger form of this corollary appears in \cite{PachecoStevensonZalesskii2009}. There instead of assuming $M\leq P_g(C)$, it is assumed that $\Pi/N$ has a 'big' $p$-sylow subgroup. However there is a gap in \cite{PachecoStevensonZalesskii2009}. In a private communication the authors of \cite{PachecoStevensonZalesskii2009} renounced their main result. The gap can be fixed under the assumption that $N\leq P_g(C)$. 
Note that under the assumption $N\leq P_g(C)$, the assertion of Corollary \ref{cor:LD} follows from Lubotzky-v.d.~Dries' theorem for free groups because $P_g(C)$ is free. But in general, it seems that one cannot apply the group theoretical theorem directly. 

If $[M:N]\neq p$ the corollary follows directly from Theorem~\ref{thm:main} Part~\eqref{part:main_b}. Indeed, 
since $M$ is open in $N$, there exists an open normal subgroup $\tilde{N}$ of $\Pi$ such that $N\cap \tilde{N}\leq M$. Thus $M$ is in a diamond with the extra properties, as needed. For the general case we need the diamond theorem for free groups,  Melnikov's theorem on normal subgroups of free groups, and Theorem~\ref{thm:main} Part\eqref{part:main_a}. See \S\ref{pf:cor}.

The proof of Theorem~\ref{thm:main} contains two key ingredients. 
We prove a geometric result on solvability of embedding problems for $\pi_1(C)$. For this we need a piece of notation. An open subgroup $\Pi^{o}$ of $\Pi$ is the \'etale fundamental group of an \'etale cover $D$ of $C$. We say $\Pi^{o}$ \emph{corresponds to a curve of genus $\g$} if the genus of the smooth completion of $D$ is $\g$. 

\begin{theorem}\label{thm:pi_1}
 Let $k$ be an algebraically closed field of characteristic $p > 0$, 
 let $C$ be a smooth affine $k$-curve, $\Pi = \pi_1(C)$, $g$ a positive integer, and $\mathcal{E}=(\mu\colon \Pi \to G, \alpha \colon \Gamma\to G)$ a finite embedding problem for $\Pi$.  Then there exists an open normal index $p$ subgroup $\Pi^o$ of $\Pi$ such that:
 \begin{enumerate}
  \item $\mu(\Pi^o) = G$.
  \item $\Pi^o$ corresponds to a curve of genus at least $\g$.
  \item The restricted embedding problem $(\mu|_{\Pi^o}: \Pi^o \to G, \alpha: \Gamma\to G)$ is properly solvable.
 \end{enumerate}
\end{theorem}

Harbater and Stevenson prove that $\Pi=\pi_1(C)$ is almost $\omega$-free, i.e., there exists an open normal subgroup $\Pi^{o}$ satisfying (1) and (3) \cite[Theorem~6]{HarbaterStevenson2011}, however no bound on the index of $\Pi^{o}$ is given. The proof of \cite[Theorem~6]{HarbaterStevenson2011} is based on ``adding branch points''. In \cite{Jarden2011} Jarden proves that a profinite group $\Gamma$ is almost-$\omega$ free if the mild condition that $A_n^m$ is a quotient of $\Gamma$ for all sufficiently large $n$ and $m$ (here $A_n$ is the alternating group). Since $A_n^m$ is quasi-$p$, if $n\geq p$, By Abhyankar's conjecture, $A_n^m$ is a quotient of $\pi_1(C)$, hence Jarden's result implies  \cite[Theorem~6]{HarbaterStevenson2011}. We note that both Harbater-Stevenson's and Jarden's methods do not give any bound on the index of $\Pi^{o}$ in $\Pi$, which is essential for applications, e.g.\ to the proof of Theorem~\ref{thm:main}. 
Our proof relies on the constructions in \cite{Kumar2008}.

The second ingredient is the Haran-Shapiro induction, see \cite{Bary-SorokerHaranHarbater2010}.

\section*{Acknowledgments}
The authors thank David Harbater and Kate Stevenson for useful discussions. 

The first author is a Alexander von Humboldt fellow. 

\section{Solutions of embedding problems for the fundamental group}
An \emph{embedding problem} $\mathcal{E}=(\mu: \Pi \to G, \alpha: \Gamma\to G)$ for a profinite group $\Pi$ consists of profinite group epimorhisms $\mu$ and $\alpha$. We call $H=\ker\alpha$ the \emph{kernel} of the embedding problem. 
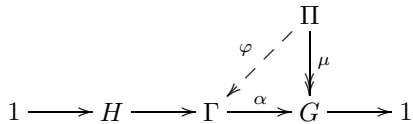
\[
 \xymatrix{
& & & \Pi\ar@{->>}[d]^{\mu}\ar@{-->}[ld]_{\phi}\\
 1\ar[r]&H\ar[r]& \Gamma\ar[r]^{\alpha}& G\ar[r]&1}
\]
The embedding problem is \emph{finite} (resp.\ \emph{split}) if $\Gamma$ is finite (resp.\ $\alpha$ splits). 
A \emph{weak solution} of $\mathcal{E}$ is a homomorphism $\phi\colon \Pi\to \Gamma$ such that $\alpha\circ\phi=\mu$. A \emph {(proper) solution} is a surjective weak solution $\phi$.

Let $k$ be an algebraically closed field of characteristic $p>0$. Let $C$ be a smooth affine $k$-curve, $\Pi = \pi_1(C)$, and $ (\mu: \Pi \to G, \alpha: \Gamma\to G)$ a finite embedding problem for $\Pi$. Let $H$ be the kernel of $\alpha$. 

\subsection{Prime-to-$p$ kernel}
In the rest of the section we shall prove the existence of solutions of finite embedding problems with prime-to-$p$ kernels when restricting to a normal subgroup of index $p$:

\begin{proposition}\label{pro:prime-to-p}
Let $g$ be a positive integer and let $\mathcal{E}=(\mu: \Pi \to G, \alpha: \Gamma\to G)$ be a finite split embedding problem for $\Pi$ with $H=\ker\alpha$ prime-to-$p$. Then there exists an open normal subgroup $\Pi^o$ of $\Pi$ of index $p$ such that
 \begin{enumerate}
  \item $\mu(\Pi^o) = G$;
  \item $\Pi^o$ corresponds to a curve of genus at least $\g$;
  \item The restricted embedding problem,  $ (\mu|_{\Pi^o}: \Pi^o \to G, \alpha: \Gamma\to G)$, is properly solvable.
 \end{enumerate}
\end{proposition}

Let $K^{un}$ denote the compositum (in some fixed algebraic closure of $k(C)$) of the function fields of all finite \'etale covers of $C$. Let $X$ be the smooth completion of $C$. 
By a strong version of Noether normalization theorem (\cite[Corollary 16.18]{Eis}), there exists a finite surjective $k$-morphism $\Phi_0: C\to \Aff^1 _x$ which is generically separable. Here  $x$ denotes the local coordinate of the affine line. The branch locus of $\Phi_0$ is of codimension $1$, hence $\Phi_0$ is \'etale away from finitely many points. By translation we may assume none of these points map to $x=0$. $\Phi_0$ extends to a finite surjective morphism $\Phi_X:X\rightarrow \PP^1 _x$.
Let $\{r_1,\ldots,r_N\}=\Phi_X^{-1}(\{x=0\})$, then $\Phi_X$
is \'etale at $r_1,\ldots,r_N$.

We use the construction in \cite[Section 6]{Kumar2008}, which we recall for the reader's convenience.  Let $\Phi_Y:Y\to \PP^1_y$ be a $p$-cyclic cover between smooth curves, locally given by $Z^p-Z-y^{-r}=0$, for some $r$ that is prime to $p$. The genus of $Y$ is 
\begin{equation}\label{eq:gY}
g_Y = (p-1)(r-1)/2 
\end{equation}
and $\Phi_Y$ is totally ramified at $y=0$ and \'etale elsewhere (see \cite{Pries2006}). Let $F$ be the zero locus of $t-xy$ in $\PP^1_x\times_{\spec(k)}\PP^1_y\times_{\spec(k)}\spec(k[[t]])$. Let $X_F$ and $Y_F$ be the normalization of an irreducible dominating component of the product $X\times_{\PP^1_x} F$ and $Y\times_{\PP^1_y} F$, respectively. Let $T$ be the normalization of an irreducible dominating component of $X_F\times_F Y_F$. We summarize the situation in Figure~\ref{defnofT}.

\begin{figure} [htbp]\renewcommand{\thefigure}{\arabic{section}.\arabic{figure}}\label{defnofT}
\begin{equation*}
\xymatrix {
        &         & T\ar@<1ex>[dl] \ar@<1ex>[dr]\\
 &X_F\ar@<1ex>[dl]\ar@<1ex>[dr] &  &Y_F\ar@<1ex>[dl]\ar@<1ex>[dr]\\
X\ar@<1ex>[dr]&       &F\ar@<1ex>[dl]\ar@<1ex>[dr]&      &Y\ar@<1ex>[dl] \\
       &\PP^1 _x &             &\PP^1 _y \\
}
\end{equation*}
\caption{}
\end{figure}

\begin{lemma}\label{lm:p-cyclic}
 The fiber of the morphism $T\to X_F$ over $\spec(k((t))$ induces a $\ZZ/p\ZZ$-cover $T\times_{\spec(k[[t]])} \spec(k((t)))\to X\times_{\spec(k)}\spec(k((t)))$ which is \'etale over $C$.
\end{lemma}

\begin{proof}
 Note that $k(X_F)=k((t))k(X)$ because generically $t\ne 0$ and over $t\ne 0$ the morphism $F\to \PP^1_x$ is  the base change $\PP^1_x\times_{\spec(k)}\spec(k((t)))\to \PP^1_x$. Also note that over $t \ne 0$ the local coordinates $x$ and $y$ satify the relation $y=t/x$ on $T$. Therefore, since $Y_F$ is the base change of $Y\to \PP^1_y$, it follows that  $k(Y_F)=k((t))(x)[Z]/(Z^p-Z-(x/t)^r)$. 
 Since $T$ is a dominating component of $X_F\times_F Y_F$, we have 
 \[
  k(T) = k(X_F)k(Y_F) = k((t))k(X)(k(x)[Z]/(Z^p-Z-(x/t)^r) ).
 \]
 But $Z^p-Z-(x/t)^r$ is irreducible over $k(X)k((t))$, hence $k(T)/k(X_F)$ is a Galois extension with Galois group  $\ZZ/p\ZZ$. Finally since the cover $T\times_{\spec(k[[t]])} \spec(k((t)))\to X\times_{\spec(k)}\spec(k((t)))$ is given by the equation $Z^p-Z-(x/t)^r=0$, it is \'etale away from $x=\infty$. In particular, it is \'etale over $C$.
\end{proof}

Since $\Gal(K^{un}/k(C))=\pi_1(C)$, the surjection $\mu\colon\pi_1(C)\to G$ induces an irrdeucible normal $G$-cover $\Psi_X:W_X\to X$ which is \'etale over $C$. By \eqref{eq:gY}, we can choose a prime-to-$p$  integer $r$ to be sufficiently large so that $g_Y \geq \max\{|H|,g\}$. Then there exist an irreducible smooth \'etale $H$-cover $\Psi_Y:W_Y\to Y$, since $H$ is a prime-to-$p$ group (\cite[XIII,Corollary 2.12]{SGA1}).

\begin{proposition}\label{many-covers}
There exist an irreducible normal $\Gamma$-cover $W\to T$ of $k[[t]]$-curves, a regular finite type $k[t,t^{-1}]$-algebra $B$ contained in $k((t))$ and an open subset $S$ of $\spec(B)$ such that the $\Gamma$-cover $W\to T$ descends to a $\Gamma$-cover of connected normal projective $\spec(B)$-schemes $W_B\to T_B$ and for any closed point $s$ in $S$ we have the following:
\begin{enumerate}
 \item The fiber $T_s$ of $T_B$ at $s$ is a smooth irreducible $\ZZ/p\ZZ$-cover of $X$ \'etale over $C$.
 \item The fiber $W_s\to T_s$ of $W_B\to T_B$ at $s$ is a smooth irreducible $\Gamma$-cover such that the composition $W_s\to T_s \to X$ is \'etale over $C$.
 \item The quotient $W_s/H$ is isomorphic to an irreducible dominating component of $W_X\times_X T_s$ as $G$-covers of $T_s$. In particular, $k(W_X)$ and $k(T_s)$ are linearly disjoint over $k(X)$. 
\end{enumerate}
\end{proposition}

\begin{proof} 
 By \cite[Proposition 6.4]{Kumar2008} there exists a normal irreducible $\Gamma$-cover $W\to T$ of $k[[t]]$-curves with prescribed properties. We claim that this cover is \'etale away from the points lying above $x=\infty$ in $T$.  
 
 Indeed, by  Conclusion $(1)$ of \cite[Proposition 6.4]{Kumar2008}, we have 
 \begin{equation}\label{eqqq}
   W\times_T \widetilde{T_X} = {\rm Ind}_G^\Gamma (\widetilde{W_{XT} \times_T T_X}),
 \end{equation}
where tilde denotes the $(t)$-adic completion, $T_X = T\setminus \{x=0\}$, $W_{XT}$ is the normalization of an irreducible dominating component of $W_X\times_X T$, and for any variety $V$, ${\rm Ind}_G^\Gamma V= V^{\Gamma}/\sim$, where $(\gamma,v)\sim (\gamma',v')$ if and only if $\gamma'=\gamma g^{-1}$ and $v'= gv$, for some $g\in G$. Note that the left hand side of \eqref{eqqq} is the $(t)$-adic completion of $W\times_T T_X$  hence the branch locus of $W\times_T T_X \to T_X$ and $W\times_T \tilde T_X\to \tilde T_X$ are same. On the other hand, the right hand side of \eqref{eqqq} is the disjoint union of copies of $\widetilde{W_{XT} \times_T T_X}$, so the branch locus of $W\times_T \tilde T_X \to \tilde T_X$ maps to the branch locus of $W_X\to X$ under the morphism $T_X\to X$. Thus, since $W_X\to X$ is \`etale away from $\{x=\infty\}$, so is $W\times_T T_X \to T_X$. Similarly, using (1') of \cite[Proposition 6.4]{Kumar2008} and the fact that $W_Y\to Y$ is \'etale everywhere, we get that over $T_Y=T\setminus \{y=0\}$ the cover $W\to T$ is \'etale. It follows from Conclusion (2) of \cite[Proposition 6.4]{Kumar2008} that $W\to T$ is \'etale over $\{x=y=0\}$, as needed.

By Conclusion (5) of \cite[Proposition 6.4]{Kumar2008} $W/H \cong W_{XT}$ as $G$-covers of $T$, hence we identify them. By Lemma~\ref{lm:p-cyclic}, away from $t=0$, $T\to X\times_k \spec(k[[t]])$ is a $\ZZ/p\ZZ$-cover.

Now applying ``Lefschetz's type principle'', \cite[Proposition 6.9]{Kumar2008}, to the proper surjective morphisms of projective $k[[t]]$-curves
 $$
 W\to W_{XT} \to T \to X\times_k\spec(k[[t]])
 $$ 
and the points $r_1,\ldots,r_N$ in $X$ lying over $x=\infty$
 it can be seen that there exists an open subset $S$ of the spectrum of a $k[t, t^{-1}]$-algebra $B$ such that the above morphisms descend to the morphisms of $\spec(B)$-schemes 
 $$
 W_B\to W_{XT,B} \to T_B \to X\times_k\spec(B)
 $$ 
 and the fiber over every point $s$ in $S$ leads to covers of smooth irreducible curves $W_s\to T_s$ and $T_s\to X$ with the desired ramification properties and Galois groups. This proves Conclusions (1) and (2). For (3), we note that $W_s\to T_s$ factors through $W_{XT,s}$ and since $W/H$ is isomorphic to $W_{XT}$, $W_s/H$ is isomorphic to $W_{XT,s}$. So $W_{XT,s}$ is an irreducible $G$-cover of $T_s$. Also $W_{XT,s}$ is an irreducible dominating component of $W_X\times_X T_s$. Finally, the compositum $k(W_X)k(T_s)$ is the function field of $W_{XT,s}$ and it is a Galois extension of $k(T_s)$ with Galois group $G$. Also $\Gal(k(W_X)/k(X))=G$, so $k(T_s)$ and $k(W_X)$ are linearly disjoint over $k(X)$.
 
 Note that the statement of \cite[Proposition 6.9]{Kumar2008} only asserts the existence of the fibers $W_s$, $T_s$, etc.\ and covering morphisms between them with the desired Galois group and ramification properties. Since the morphisms of $k[[t]]$-curves $W\to T$ and $T\to X_F$ are finite, they descend to morphisms of a regular finite type $k[t,t^{-1}]$-algebra $B$ contained in $k((t))$. In the proof of \cite[Proposition 6.9]{Kumar2008} it is further shown that there exist an open subset of $\spec(B)$ such that fibers over any point in this open subset have the desired properties.  
\end{proof}

\begin{proof}[Proof of Proposition \ref{pro:prime-to-p}]
 Using the notation of the above proposition, let $\Sigma \subset T_s$ be the preimage of $X\setminus C$ under the covering $T_s\to X$, let $T_s^o=T_s\setminus \Sigma$, and let $\Pi^o=\pi_1(T_s^o)$. By $(1)$ of Proposition \ref{many-covers}, 
 $k(T_s^o)/k(C)$ is a Galois extension with Galois group $\ZZ/p\ZZ$. If $Z\to T_s^o$ is an \'etale cover of $T_s^o$ then the composition with $T_s^o\to C$ gives an \'etale cover $Z\to C$. Conversely, if $Z\to C$ is an \'etale cover of $C$ which factors through $T_s^o\to C$, then $Z\to T_s^o$ is also \'etale. So $\Pi^o=\Gal(K^{un}/k(T_s))=\pi_1(T_s^o)$ is an index $p$ normal subgroup of $\Gal(K^{un}/k(C))=\pi_1(C)$.

 For (2), note that by construction of $T$ (see Figure~\ref{defnofT}), $T$ dominates $Y_F$, so over the generic point, the genus of the $k((t))$-curve $T\times_{\spec(k[[t]]}\spec(k((t))$ is at least the genus of $Y$. But $g_Y\ge \g$, so this genus is at least $g$. Since $T_s$ is a smooth fiber of $T_B$ which in turn descends from $T$, the genus of $T_s$ is also at least $\g$.    

 Since the $G$-cover $W_X\to X$ is induced by the epimorphism $\mu\colon \Pi\to G$, we have $\ker\mu=\Gal(K^{un}/k(W_X))$. So $(\ker\mu)\cap \Pi^o=\Gal(K^{un}/k(W_X)k(T_s))$. Hence 
 $$
 \mu(\Pi^o)=\Pi^o/((\ker\mu)\cap\Pi^o)=\Gal(k(W_X)k(T_s)/k(T_s)).
 $$
  But $\Gal(k(W_X)k(T_s)/k(T_s))=G$ by $(3)$ of Proposition \ref{many-covers}, so we get $\mu(\Pi^o) = G$, as needed.

 Also from Proposition \ref{many-covers}, we get a $\Gamma$-cover $W_s$ of $T_s$ such that $k(W_X)\subset k(W_s)$ and we obtain the following tower of field extensions:
$$
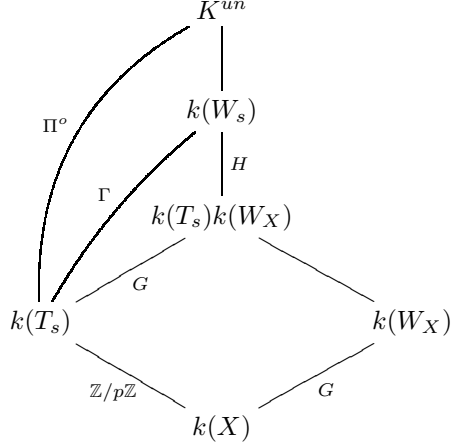
\xymatrix{
  & K^{un}\ar@{-}[d]\ar@/_2pc/@{-}[dddl]_{\Pi^o} \\
  & k(W_s)\ar@{-}[d]^H\ar@/_/@{-}[ddl]_{\Gamma}\\
  & k(T_s)k(W_X)\ar@{-}[dl]^G\ar@{-}[dr]\\
  k(T_s)\ar@{-}[dr]_{\ZZ/p\ZZ}& &k(W_X)\ar@{-}[dl]^{G}\\
  & k(X) }
$$ 

Hence there is a surjection from $\Pi^o \onto
\Gamma=\Gal(k(W_s)/k(T_s))$ which dominates $\mu|_{\Pi^o}\colon\Pi^o
\onto G$. Hence $W_s$ provides a solution to the embedding problem restricted to $\Pi^o$.
\end{proof}

\subsection{Proof of Theorem~\ref{thm:pi_1}}
\begin{trivlist}
 \item Let $k$ be an algebraically closed field of characteristic $p > 0$, 
 let $C$ be a smooth affine $k$-curve, $\Pi = \pi_1(C)$, $g$ a positive integer, and $\mathcal{E}=(\mu\colon \Pi \to G, \alpha \colon \Gamma\to G)$ a finite embedding problem for $\Pi$. Put $H=\ker\alpha$.
 
 Since $\Pi$ is projective (\cite[Proposition 1]{Serre1990}), $\mathcal{E}$ has a weak solution. It is well known that $\mathcal{E}$  can be  dominated  by a finite split embedding problem $\mathcal{E}'$, in the sense that a solution of $\mathcal{E}'$ induces a solution of $\mathcal{E}$ (even when restricted to a subgroup $\Pi^o$ with $\mu(\Pi^o)=G$), see for example \cite[Discussion after Lemma~2.3]{HarbaterStevenson2005}. Hence, without loss of generality, we can assume that $\mathcal{E}$ splits. 
 
 Let $p(H)$ be the subgroup of $H$ generated by all $p$-sylow subgroups. Since $p(H)$ is characteristic in $H$, it is normal in $\Gamma$. Then there is $\alpha'\colon \Gamma/p(H)\to G$ such that $\alpha = \alpha'\circ \alpha''$, where $\alpha''\colon \Gamma\to \Gamma/p(H)$ is the natural quotient map. 
 
 Since $\ker \alpha'$ is prime-to-$p$, by Proposition \ref{pro:prime-to-p}, there exists an index $p$ normal subgroup $\Pi^o$ of $\Pi$ such that $\mu(\Pi^{o})=G$, $\Pi^{o}$ corresponds to a curve of genus at least $g$, and the embedding problem  
 \[
  (\mu|_{\Pi^o}: \Pi^o \to G, \alpha'\colon \Gamma/p(H)\to G)
 \]
has a solution, say $\mu'\colon \Pi^{o}\to \Gamma/p(H)$. 

The kernel of $\alpha''$ is $p(H)$ which by definition is quasi-$p$. So using Pop's result (\cite[Theorem B]{pop}) the embedding problem 
\[
(\mu': \Pi^o \to \Gamma/p(H), \alpha''\colon  \Gamma \to\Gamma/p(H)) 
\]
 has a solution, say $\phi$. Then 
 \[
  \alpha\circ \phi = \alpha'\circ \alpha''\circ \phi=\alpha'\circ \mu'=\mu|_{\Pi^o}.
 \]
Hence $\phi$ is also a solution of $ (\mu|_{\Pi^o}: \Pi^o \to G, \alpha\colon \Gamma\to G)$, as needed.\qed
\end{trivlist}

\section{Proof of Theorem~\ref{thm:main}}
 Let $k$ be a countable algebraically closed field of characteristic $p>0$, let $C$ be an affine $k$-curve, let $\Pi=\pi_1(C)$, let $g$ be a non-negative integer, and let $M\leq P_g(C)$ be a subgroup of $\Pi$. In particular $[\Pi:M]=\infty$. Assume there exist normal subgroups $M_1,M_2$ of $\Pi$ such that $M_1\cap M_2\leq M$ and $M_1,M_2\not\leq M$. In Part~\eqref{part:main_a} we have to prove that for every finite simple group $S$, the direct power $S^\infty$ is a quotient of $M$. In Part~\eqref{part:main_b} we further assume that $[MM_i:M]\neq p$, and we have to prove that $M$ is free profinite group of countable rank.  
 
 \subsection*{First reduction} We can assume that $[MM_1:M] = \infty$ and $M\neq MM_2$. For Part~\eqref{part:main_b} we can further assume that $[MM_2:M]\neq p$.
 
 Indeed, if $[M_1:M_1\cap M]= [MM_1:M]<\infty$, then there exists an open normal subgroup $U$ of $\Pi$ such that $U\cap M_1\leq M\cap M_1 \leq M$. Since $[\Pi:M]=\infty$, we have $[U:U\cap M]=[UM:M]=\infty$. Thus we can replace $M_1$ by $U$ and $M_2$ by $M_1$. 
 
 \subsection*{Second reduction}  For Part~\eqref{part:main_b} it suffices to solve an arbitrary finite split embedding problem
\[
\mathcal{E}_1 = (\mu \colon M \to G_1, \alpha_1\colon A\rtimes G_1 \to G_1)
\]
for $M$ and for Part~\eqref{part:main_a} it suffices to solve $\mathcal{E}_1$ when $G_1=1$ and $A=S^n$, for arbitrary $n\geq 1$. 

Indeed, since $\Pi$ is projective \cite[Proposition 1]{Serre1990}, $M$ is also projective \cite[Proposition 22.4.7]{FriedJarden2008}. Since $k$ is countable, $\Pi$ is of rank $\aleph_0$, hence ${\rm rank}(M)\leq \aleph_0$. Thus to prove that $M$ is free of countable rank it suffices to properly solve any finite split embedding problem for $M$ \cite[Theorem 2.1]{HarbaterStevenson2005}. 
In order to prove that $S^\infty$ is a quotient of $M$, it suffices to prove that $S^n$ is a quotient for every positive integer $n$. The latter is equivalent to properly solve any $\mathcal{E}_1$ with $G_1=1$ and $A=S^n$.
 
\subsection*{Solving $\mathcal{E}_1$ under the assumption $\boldsymbol{[MM_2:M]\neq p}$}
Let $L$ be an open normal subgroup of $\Pi$ such that $M\cap L \leq \ker\mu$ and the following conditions hold: 
\begin{enumerate}\renewcommand{\theenumi}{\Alph{enumi}}
\item \label{it:1}
 $[M_1 M L : M L ] = [M_1 : M_1\cap ( ML) ] \geq 3 p$ (recall $[M_1M:M]=\infty$).
\item \label{it:2}
 $[M_2ML:ML] = [ M_2 : M_2\cap ( ML)] \neq 1,p$ (recall $[M_2M:M]\neq 1,p$) .
\item \label{it:3}
 $[\Pi:ML] \geq 3p$ (recall $[\Pi:M]=\infty$).
\end{enumerate}

Let $G = \Pi/L$, $\phi\colon \Pi\to G$ the natural epimorphism, and $G_0 = \phi(M)= ML/L$. Since $M\cap L\leq \ker \mu$, the map $\mu$ factors as $\bar \mu\circ \phi|_M $, $\bar\mu \colon G_0\to G_1$. Hence $G_0$ acts on $A$ via $\bar\mu$, namely $a^g := a^{\bar\mu(g)}$, $a\in A$, $g\in G_0$. Let $I= \{ f\colon G\to A \mid f(\sigma\rho)=f(\sigma)^\rho, \ \forall \sigma\in G, \rho\in G_0\}$.  
Then $I\cong A^{(G:G_0)}$ and $G$ acts on $I$ by the formula 
\[
 f^\sigma(\tau) = f(\sigma \tau), \qquad \sigma,\tau\in G.
\]
We denote the corresponding semidirect product by $A\wr_{G_0} G := I\rtimes G$, and we refer to it as the twisted wreath product. It comes with a canonical projection map $\alpha\colon A\wr_{G_0} G \to G$. Thus 
\[
\mathcal{E} = (\phi\colon \Pi\to G, \alpha\colon A\wr_{G_0} G\to G)
\]
is a split embedding problem for $\Pi$. Theorem~\ref{thm:pi_1} gives an index $p$ open normal subgroup $U$ of $ \Pi$ containing $P_g(C)$, hence $M\leq U$, such that $\phi(U) = G$, hence $UL=\Pi$; and such that the restricted embedding problem 
\begin{equation}\label{eq:mathcalEU}
 \mathcal{E}|_U = (\phi|_U\colon U\to G, \alpha\colon A\wr_{G_0} G\to G)
\end{equation}
is solvable. Let $\psi\colon U \to A\wr_{G_0} G$ be a solution of $\mathcal{E}|_U$. 

\[
\xymatrix{
M_i\cap U \ar@{-}[r]^{\leq p}\ar@{-}[d]\ar@<-7ex>@/_10pt/@{--}[dd]_{A_i}							& M_i\ar@{-}[d]\\
M_i\cap (ML)\cap U\ar@{-}[r]\ar@{-}[d]					&M_i\cap (ML)\ar@{-}[dd]\\
(M_i\cap U) \cap (M  (L\cap U))\ar@{-}[d]\\
M_i\cap U \cap M \ar@{=}[r]								&M_i\cap M
}
\]

Let $A_i = [M_i\cap U :(M_i\cap U)\cap (M(L\cap U))]$, for $i=1,2$. Then $A_1 \geq \frac{[M_1:M_1\cap (ML)]}{p} \geq 3$. 
If $M_2\leq U$, then $M_2 \cap U = M_2$, and we have 
\[
A_2=[M_2 :M_2\cap (M(L\cap U))] \geq [M_2:M_2\cap ML] >1.
\]
If $M_2\not\leq U$, then $[M_2:M_2\cap U] = p$. Then we have
\begin{eqnarray*}
p\cdot [M_2\cap U:M_2\cap (ML)\cap U] &=&[M_2:M_2\cap U] [M_2\cap U:M_2\cap (ML)\cap U] \\
 &=&[M_2:M_2\cap (ML)\cap U] \\
 &=& [M_2:M_2\cap(ML)][M_2\cap(ML):M_2\cap(ML)\cap U]. 
\end{eqnarray*}
By \eqref{it:2} we get that the right hand side does not equal $p$, hence $[M_2\cap U:M_2\cap (ML)\cap U] > 1$, hence $A_2>1$. 
Since $UL = \Pi$, we have 
\[
[U:M (L\cap U)] \geq [U:U\cap ML] = [UML:ML] = [\Pi:ML]\geq 3p\geq 3.
\]

The last paragraph implies that if we set $\tilde{\Pi}=U$, $\tilde{M}_i=M_i\cap U$, $\tilde{M}=M\cap U = M$, $\tilde{L}=L\cap U$, and $\tilde{\mathcal{E}}=\mathcal{E}|_{U}$, then we have 
\begin{enumerate}\renewcommand{\theenumi}{\~{\Alph{enumi}}}
\item \label{it:1tilde}
 $[\tilde{M}_1 \tilde M \tilde L : \tilde M \tilde L ] = [\tilde M_1 : \tilde M_1\cap ( \tilde M \tilde L) ] \geq 3$.
\item \label{it:2tilde}
 $[\tilde M_2 \tilde M\tilde L:\tilde M\tilde L] = [ \tilde M_2 : \tilde M_2\cap ( \tilde M\tilde L)] >1$ .
 \item \label{it:3tilde}
 $[\tilde \Pi:\tilde M\tilde L] \geq 3$.
\end{enumerate}
and a proper solution $\psi$ of $\tilde{\mathcal{E}}$. 

Let $K_i = \phi(\tilde{M}_i)=\tilde{M}_i\tilde{L}/\tilde{L}$, $i=1,2$. Then \eqref{it:1tilde}-\eqref{it:3tilde} can be reformulated as
\begin{enumerate}\renewcommand{\theenumi}{\alph{enumi}}
\item $[K_1G_0:G_0]\geq 3$;
\item $[K_2 G_0:G_0]\geq 2$;
\item $[G:G_0]\geq 3$. 
\end{enumerate}

We apply \cite[Proposition~4.6]{Bary-SorokerHaranHarbater2010}. Let $D=\tilde{L}$, $\tilde{\Pi}_0 = \tilde{M}\tilde{L}$, and $\tilde{N}=\tilde{L}\cap \tilde{M_1}\cap \tilde{M_2}$. Let $A_0$ be a proper subgroup of $A$ that is normal in $A\rtimes G_0$. Then $G_0$ acts  on the nontrivial group $\bar{A}$. It suffices to prove by  \cite[Proposition~4.6]{Bary-SorokerHaranHarbater2010} that in this setting the embedding problem 
\[
(\bar \phi \colon \tilde{\Pi}/N\to G, \bar{\alpha} \colon \bar{A}\wr_{G_0} G\to G)
\]
is not solvable. Here $\bar\phi$ is the quotient map (recall that $G\cong \tilde{\Pi}/\tilde{L}$ and $N\leq \tilde{L}$). 

Put $H = \bar{A} \wr_{G_0} G$. 
Assume by negation that there is an epimorphism $\rho\colon \tilde{\Pi}\to H$ with $\bar\alpha \circ \rho = \phi$ and $\rho(N) = 1$. For $i=1,2$, let $H_i = \rho(\tilde{M}_i)$. Then $H_i$ is normal in $H$ and $\alpha(H_i) = \phi(\tilde{M_i}) = K_i$. By \cite[Lemma 13.7.4(a)]{FriedJarden2008} there is $h_1\in H_1$ and $h_2\in H_2$ such that $\alpha(h_1) = 1$ and $[h_1, h_2]\neq 1$. 

Let $\gamma_i \in \tilde{M}_i$ with $\rho(\gamma_i) = h_i$. Then $\phi(\gamma_1) = \alpha(h_1) = 1$, so $\gamma_1\in L$. Then 
\[
[\gamma_1,\gamma_2]\in [\tilde{L},\tilde{M}_2]\cap [\tilde{M}_1,\tilde{M}_2]\leq \tilde{L}\cap (\tilde{M}_1\cap \tilde{M}_2) = N.
\]
(Recall that if $A,B$ are normal subgroups of a group $C$, then $[A,B]\leq A\cap  B$, because $a^{-1}b^{-1} a b = a^{-1}  a^{b} = b^{-a} b$.)
We thus get that $[h_1,h_2] = [\rho(\gamma_1),\rho(\gamma_2)] \in \rho(N) = 1$. This contradiction implies there is no proper solution $\rho$ of $(\bar\phi, \bar\alpha)$, and hence by \cite[Proposition~4.6]{Bary-SorokerHaranHarbater2010} $\psi$ induces a proper solution $\nu$ of $\mathcal{E}_1$. This finishes the proof of Part~\eqref{part:main_b} by the reduction steps. 

\subsection*{Solving $\mathcal{E}_1$ with $G_1=1$ and $A=S^n$}

We separate the proof into two cases.
\subsubsection*{Case A} Assume $p\mid |S|$.  Then $S$ is a quasi-$p$ group and hence $A=S^n$ is quasi-$p$.  Then by Pop's result \cite[Theorem~B]{pop} we get that the embedding problem $(\Pi\to \Pi/L, \alpha\colon A\wr_{G_0} G \to G)$ is properly solvable; let $\psi\colon \Pi \to A\wr_{G_0} G$ be a proper solution. If we take $U=\Pi$ above instead of the index $p$ subgroup $\tilde{\Pi}$ of $\Pi$, then we have $\tilde{\mathcal{E}}=\mathcal{E}$, $\tilde{M}=M$, etc. In particular \eqref{it:1tilde}-\eqref{it:3tilde} follows directly from \eqref{it:1}-\eqref{it:3}. Then the rest of the proof can be applied mutatis mutandis to get that $\psi$ induces a solution $\nu$ of $\mathcal{E}_1$, as needed by the reduction steps. 

\subsubsection*{Case B} Assume $p\nmid |S|$. 
We note that the above proof uses the assumption $[M_2:M_2\cap M] =[M_2 M:M]\neq p$ only to ensure that \eqref{it:2} is satisfied which in turn is only used when $M_2\not\leq U$. Thus if $M_2\leq U$, the above can be applied to get that $\mathcal{E}_1$ is properly solvable, and we are done. 

We thus assume that $[M_2:M_2\cap M] =[M_2 M:M]= p$ and $M_2\not\leq U$. Thus, since $[\Pi:U]=p$ we have $[M_2:M_2\cap U]=p$. Since $M\leq P_g(C)\leq U$ and since $[M_2:M_2\cap M]=p$ we have $M_2\cap M =M_2\cap U$. In particular $U/(M_2\cap U) \cong \Pi/M_2$ hence $U/M_2$ is a quotient of $\Pi$.

Let $X$ be a completion of $C$, let $g_{\Pi}$ be the genus of $X$ and $r+1 = |X\smallsetminus C|$. Let $\nu$ be the maximal positive integer such that $S^\nu$ is a generated by $2g+r$ elements. Then $\nu$ is the maximal integer such that $S^\nu$ is a quotient of $\Pi$, as a prime-to-$p$ group. 

We note that in  the choice of $U$, the genus $g_U$ of the completion of the corresponding curve can be taken such that $S^{\nu+n}$ is generated by $2g_U$ elements. Since the prime-to-$p$ quotient of $U$ is free of rank at least $2 g_U$, there exists an epimorphism $\psi \colon U \to S^{\nu+n}$. Let $\Gamma = \psi(M_2\cap U)$. 

Since $M_2\cap U\lhd U$ it follows that $\Gamma \normal S^{\nu+n}$. Thus $\Gamma\cong S^{\nu_1}$ and $S^{\nu+n} = \Gamma\times \Lambda$, where $\Lambda\cong S^{\nu_2}$ and $\nu_1+\nu_2=\nu+n$. We have $\Lambda \cong \psi(U/(M_2\cap U))$, hence $\Lambda$ is a quotient of $U/(M_2\cap U)\cong \Pi/M_2$, and hence of $\Pi$. Thus $\nu_2\leq \nu$, so $\nu_1\geq  n$. 

Let $\phi\colon U \to \Gamma$ be the composition of $\psi$ with the projection onto $\Gamma$. Then $\phi(M_2\cap U)=\psi(\Gamma\times 1)=\Gamma$. Since $M_2\cap U \leq M \leq U$ we get that $\phi(M) = \Gamma$. This finishes the proof because $\Gamma\cong S^{\nu_1}$ with $\nu_1\geq n$. 
\qed

\subsection{Proof of Corollary~\ref{cor:LD}}\label{pf:cor}

 Let $k$ be a countable algebraically closed field of characteristic $p>0$, let $C$ be an affine $k$-curve, let $\Pi=\pi_1(C)$. Let $N$ be a normal subgroup of infinite index and let $M$ be a proper open subgroup of $N$ of index $\neq p$ that is contained in $P_g(C)$ for some $g\geq 0$. We have to prove that $M$ is free. 
 
 Let $U$ be an open normal subgroup of $\Pi$ such that $N\cap U\leq M$. Then $M$ is in the $\Pi$-diamond determined by $N$ and $U$. Note that $[\Pi:M]=\infty$, so $[UM:M]=\infty$. Thus if $[N:M]\neq p$, then $M$ is free by Part~\eqref{part:main_b} of Theorem~\ref{thm:main} and we are done. 
 
 If $N\leq P_g(C)$, then $M$ is in the $P_{g}(C)$-diamond determined by $U\cap P_{g}(C)$ and $N$. Hence $M$ is free by the diamond theorem for free profinite groups.
 
 If $N\not\leq P_g(C)$ and $[N:M]=p$, then $M=P_g(C)\cap N$ (because $M\leq P_g(C)\cap N \lneqq N$). Thus $M$ is a normal subgroup of the free profinite group $P_g(C)$. By Part~\eqref{part:main_a} of Theorem~\ref{thm:main}, $S^\infty$ is a quotient of $M$, for every finite simple group $S$. Thus by Melnikov's theorem $M$ is free. \qed

\bibliographystyle{amsplain}
\providecommand{\bysame}{\leavevmode\hbox to3em{\hrulefill}\thinspace}
\providecommand{\MR}{\relax\ifhmode\unskip\space\fi MR }
\providecommand{\MRhref}[2]{%
  \href{http://www.ams.org/mathscinet-getitem?mr=#1}{#2}
}
\providecommand{\href}[2]{#2}

\end{document}